\documentclass[a4paper,12pt]{article}

\usepackage{amsmath,amsfonts,amssymb}
\usepackage{amsthm}
\usepackage{color,multicol}

\newtheorem{lemma}{Lemma}[section]
\newtheorem{theorem}[lemma]{Theorem}

\newtheorem{proposition}[lemma]{Proposition}

\newtheorem{remark}[lemma]{Remark}

\begin{document}

\title{On the limit behaviour of second order relative spectra of  self-adjoint operators}

\author{Eugene Shargorodsky\footnote{E-mail: \ eugene.shargorodsky@kcl.ac.uk} \\
Department of Mathematics, King's College London,\\
Strand, London WC2R 2LS, UK}

\date{}

\maketitle

\begin{abstract}
It is well known that the standard projection methods allow one to recover the
whole spectrum of a bounded self-adjoint operator but they often lead to spectral pollution,
i.e. to spurious eigenvalues lying in the gaps of the essential spectrum. 
Methods using second order relative spectra are free from this problem, 
but they have not been proven to approximate the whole spectrum. L. Boulton (2006, 2007)
has shown that second order relative spectra approximate all isolated eigenvalues
of finite multiplicity. The main result of the present paper is that second order relative spectra 
do not in general approximate the whole of the essential spectrum of a 
bounded self-adjoint operator.
\end{abstract}

\section{Introduction}

Let $\mathcal{H}$ be a Hilbert space and $\mathcal{B}(\mathcal{H})$ be the space of bounded
linear operators on $\mathcal{H}$. Let
$\mathcal{L}_1 \subset \mathcal{L}_2 \subset \cdots \subset
\mathcal{L}_l \subset \mathcal{L}_{l + 1} \subset \cdots$
be an increasing sequence of finite dimensional linear subspaces of
$\mathcal{H}$ such that the corresponding
orthogonal projections $P_l : \mathcal{H} \to \mathcal{L}_l$
converge strongly to the identity operator $I$. Let
$\mathfrak{P}(\mathcal{H})$ be the set of all such sequences of subspaces.

Suppose $T = T^* \in \mathcal{B}(\mathcal{H})$ and denote the spectrum of
$P_l T : \mathcal{L}_l \to \mathcal{L}_l$ by
$\mbox{Spec}(T, \mathcal{L}_l)$. Then
\begin{equation}\label{limincl}
\lim_{l \to \infty} \mbox{Spec}(T, \mathcal{L}_l) \supseteq
\mbox{Spec}(T) , 
\end{equation}
where  ``$\lim$'' is defined in an appropriate way (see, e.g., \cite{A} or \cite{S1}).
Unfortunately the left-hand side of \eqref{limincl} may be strictly larger than the right-hand
side. This is  called spectral pollution
(see, e.g., \cite{B1, B2, DP, LS, P, RSSV, S1}) which is a well known phenomenon
in numerical analysis: spurious ``eigenvalues" may appear in the gaps of the
essential spectrum of $T$ and as a result $\lim_{l \to \infty}\mbox{\rm Spec}(T, \mathcal{L}_l)$
may contain points that do not belong to $\mbox{\rm Spec}(T)$. 

A possible way of dealing with spectral pollution 
is based on the notion of second order relative spectra which was introduced by 
E.B. Davies in \cite{D}:   
$$
\mbox{\rm Spec}_2(T, \mathcal{L}_l) := \{\lambda \in \mathbb{C} : \ 
P_l(T - \lambda I)^2 : \mathcal{L}_l \to \mathcal{L}_l
\ \mbox{ is not invertible}\} .
$$
Although the spectrum of a self-adjoint operator $T$ is a subset of $\mathbb{R}$,
the set $\mbox{\rm Spec}_2(T, \mathcal{L}_l)$ may and usually does contain points
from $\mathbb{C}\setminus\mathbb{R}$. Since $T^* = T$, it is easy to see that
$\mbox{\rm Spec}_2(T, \mathcal{L}_l)$ is symmetric with respect to the real line:
$$
\lambda \in \mbox{\rm Spec}_2(T, \mathcal{L}_l) \ \ \Longleftrightarrow \ \ 
\overline{\lambda}  \in \mbox{\rm Spec}_2(T, \mathcal{L}_l) .
$$
If 
$\lambda \in \mbox{\rm Spec}_2(T, \mathcal{L}_l)$ then
\begin{equation}
\label{nonpoll}
\mbox{\rm Spec}(T) \cap [\mbox{\rm Re}\, \lambda - |\mbox{\rm Im}\, \lambda|,
\mbox{\rm Re}\, \lambda + |\mbox{\rm Im}\, \lambda|]
\not= \emptyset 
\end{equation}
(\cite{LS, S1}, see also \cite{Kat0}). This means that  if a point of 
$\mbox{\rm Spec}_2(T, \mathcal{L}_l)$ is close to the real line, then it is close to 
$\mbox{\rm Spec}(T)$, i.e. that, in a sense, second order relative spectra do not pollute.

A natural question, which was first posed in \cite{S1}
(see also \cite{LS, S2}), is whether $\mbox{\rm Spec}_2(T, \mathcal{L}_l)$, 
$(\mathcal{L}_l)_{l \in \mathbb{N}} \in \mathfrak{P}(\mathcal{H})$ capture the whole
spectrum of $T$, i.e. whether or not
$$
\lim_{l \to \infty}  \mbox{\rm Spec}_2(T, \mathcal{L}_l) \supseteq
\mbox{\rm Spec}(T) .
$$
A partial answer to this question was obtained in \cite{B1, B2}: 
$$
\lim_{l \to \infty}  \mbox{\rm Spec}_2(T, \mathcal{L}_l) \supseteq
\{\mbox{isolated eigenvalues of } T
\mbox{ of finite multiplicity}\} .
$$
The main result of the present paper is that $\mbox{\rm Spec}_2(T, \mathcal{L}_l)$ 
do not in general approximate the whole of the essential spectrum 
$\mbox{\rm Spec}_e(T)$ of $T$. In order to state the result, we need the following
notation. Let $d_H(F,G)$ denote the Hausdorff distance between two sets 
$F, G \subset \mathbb{C}$:
$$
d_H(F,G) := \max\left\{\sup_{x \in F} \inf_{y \in G} |x - y| , \,
\sup_{y \in G} \inf_{x \in F} |x - y| \right\} .
$$

Let $\Sigma \subset \mathbb{R}$ be a compact set,
\begin{eqnarray*}
m := \min\Sigma , \ M := \max\Sigma , \ \ \
[m, M] \setminus \Sigma =  \cup_j (m_j, M_j) , \\ 
(m_j, M_j)\cap (m_l, M_l) = \emptyset \ \mbox{ if } \ j \not= l .
\end{eqnarray*}
Define
$$
\mathcal{Q}(\Sigma) := B[m, M]\setminus\cup_j B(m_j, M_j) , 
$$
where $B[c_1, c_2]$ and $B(c_1, c_2)$ denote the closed and the open disk with the diameter
$[c_1, c_2]$.

\begin{theorem}\label{Main}
Let 
$$
-\infty < \rho_-^{(1)}  < \rho_+^{(1)} < \rho_-^{(2)}  <  \rho_+^{(2)}
< \cdots < \rho_-^{(n)}  < \rho_+^{(n)} < +\infty , \ \ \ n \in \mathbb{N},
$$
and let $F \subseteq \mathcal{Q}\left(\bigcup_{j = 1}^n \left[\rho_-^{(j)}, \, \rho_+^{(j)}\right]\right)$ 
be a compact set symmetric with respect to the real line and such that
\begin{equation}
\label{nonempty}
F\cap \left(\rho_-^{(j)}, \, \rho_+^{(j)}\right) \not= \emptyset , \ \ \ j = 1, \dots, n.
\end{equation}
Then there exist $T = T^* \in \mathcal{B}(\mathcal{H})$ and $(\mathcal{L}_l) \in 
\mathfrak{P}(\mathcal{H})$ such that $\mbox{\rm Spec}(T) = 
\bigcup_{j = 1}^n \left[\rho_-^{(j)}, \, \rho_+^{(j)}\right]$ and
$$
d_H\left(\mbox{\rm Spec}_2(T, \mathcal{L}_l), F\right) \to 0 \ \mbox{ as } \ l \to +\infty .
$$
\end{theorem}

Note that
$$
\bigcup_{(\mathcal{L}_l) \in \mathfrak{P}(\mathcal{H})}{\lim_{l \to +\infty}} 
\mbox{\rm Spec}_2(T, \mathcal{L}_l) = \mbox{\rm Spec}(T) \cup 
\mathcal{Q}(\mbox{\rm Spec}_e(T)) , \ \ \ \forall\, T = T^* \in \mathcal{B}(\mathcal{H}) ,
$$
where  ``$\lim$'' is defined in an appropriate way (\cite{S1}, see also \cite{BS}).

\section{Auxiliary results}\label{aux}

\begin{proposition}\label{ext}
Let $B, M \in \mathcal{B}(\mathcal{H})$, $B^* = B$, $M^* = M \ge 0$. There exist a Hilbert space
$\mathcal{H}_0 \supseteq \mathcal{H}$  and $T = T^* \in \mathcal{B}(\mathcal{H}_0)$ such that
$B = PT|_\mathcal{H}$, $M = PT^2|_\mathcal{H}$, where $P : \mathcal{H}_0 \to \mathcal{H}$ is
the orthogonal projection, if and only if
\begin{equation}
\label{cond}
B^2 \le M .
\end{equation} 
\end{proposition}
\begin{proof}
Suppose such $\mathcal{H}_0$ and $T$ exist. Then
\begin{eqnarray*}
&& (B^2x, x) = \|Bx\|^2 = \|PTx\|^2, \\
&& (Mx, x) = (PT^2x, x) = (T^2x, x) = \|Tx\|^2, \ \  \forall x \in \mathcal{H} .
\end{eqnarray*}
Hence $(B^2x, x) \le (Mx, x)$, $ \forall x \in \mathcal{H}$, i.e. \eqref{cond} holds.

Suppose now \eqref{cond} holds. Then $M - B^2 \ge 0$ has a nonnegative square root
$(M - B^2)^{1/2}$. Let $\mathcal{H}_0 := \mathcal{H}\oplus\mathcal{H}$, 
$P : \mathcal{H}_0 \to \mathcal{H}$ be the projection onto the first component and let
\begin{equation}
\label{ constr}
T := \begin{pmatrix}
  B    &  (M - B^2)^{1/2}  \\ \\
   (M - B^2)^{1/2}   &  0
\end{pmatrix} : 
\begin{array}{c}
\mathcal{H}  \\
\oplus \\
\mathcal{H}
\end{array}
\to 
\begin{array}{c}
\mathcal{H}  \\
\oplus \\
\mathcal{H}
\end{array} = \mathcal{H}_0 .
\end{equation}
Then $T^* = T$, $PT|_\mathcal{H} = B$,
$$
T^2 =
\begin{pmatrix}
  M   &  B (M - B^2)^{1/2}  \\ \\
  (M - B^2)^{1/2} B    &  M - B^2 \end{pmatrix} 
$$
and $PT^2|_\mathcal{H} = M$.
\end{proof}

\begin{lemma}\label{block}
For any $\rho_- < \rho_+ \in \mathbb{R}$, $r \in (\rho_-, \rho_+)$ and
$\delta, \varepsilon > 0$ there exist $N \in \mathbb{N}$ and
Hermitian matrices $B, R \in \mathbb{C}^{N\times N}$ such that $\|R\| < \varepsilon$,
$\mbox{\rm Spec}(B) \subset [\rho_-, \rho_+]$, the distance from any point of $[\rho_-, \rho_+]$
to $\mbox{\rm Spec}(B)$ is less than $\delta$, and all roots of the equation
\begin{equation}
\label{det}
\det (\lambda^2 I - 2 \lambda B + B^2 + R^2) = 0
\end{equation}
belong to the vertical interval $\{\lambda \in \mathbb{C} : \ \mbox{\rm Re}\, \lambda = r , \ 
|\mbox{\rm Im}\, \lambda| < \varepsilon\}$.
\end{lemma}
\begin{proof}
It is sufficient to prove the lemma for $\rho_\pm = \pm \rho$ , $\rho > 0$ as the general case can be
reduced to this one by dealing with $B - \frac{\rho_- + \rho_+}2\, I$ instead of $B$.

Let $\varepsilon_0$ be a small positive number to be specified later
and let $w$ be a conformal mapping of the unit disk onto the ellipse with the axes 
$$
[-\rho, \rho] + i\frac{\varepsilon_0}2 \ \  \mbox{ and } \ \   i[0, \varepsilon_0] ,
$$ 
such that
$\mbox{\rm Re}\, w(0) = r$.

Let $b \in C(\mathbb{T})$ and $a \in C(\mathbb{T})$ be the boundary values of 
$\mbox{\rm Re}\, w$ and $\mbox{\rm Im}\, w$ respectively. Then $b(\mathbb{T}) = [-\rho, \rho]$
and $a(\mathbb{T}) = [0, \varepsilon_0]$. 

For any $n \in \mathbb{N}$,  the $n\times n$ Toeplitz matrix $T_n(b)$ with the symbol 
$b$ is Hermitian and 
$$
\|T_n(b)\| \le \|T(b)\| = \|b\|_{\infty} = \rho ,
$$
where $T(b) : l^2 \to l^2$ is the corresponding Toeplitz operator. Hence
$$\mbox{\rm Spec}(T_n(b)) \subset [-\rho, \rho] .
$$
It follows from Szeg\"o's theorem (see, e.g., \cite[Theorem 5.10]{BG}) that the distance 
from any point of $[-\rho, \rho]$ to $\mbox{\rm Spec}(T_N(b))$ 
is less than $\delta$ provided $N$ is sufficiently large. Fix such an $N$ and set
$B := T_N(b)$, $A := \sqrt{2\rho\varepsilon_0}\, I + T_N(a) = A^*$.
 
Since $b + ia$ is the boundary value of the function $w$ analytic in the unit disk, 
$B + iA = i\sqrt{2\rho\varepsilon_0}\, I + T_N(b + ia)$ is a lower triangular matrix
with the diagonal entries equal to $i\sqrt{2\rho\varepsilon_0} + b_0 + ia_0$, where
\begin{eqnarray*}
&& b_0 := \frac{1}{2\pi} \int_0^{2\pi} \mbox{\rm Re}\, w(e^{it}) dt = 
\mbox{\rm Re}\, w(0) = r,   \\ 
&& a_0 := \frac{1}{2\pi} \int_0^{2\pi} \mbox{\rm Im}\, w(e^{it}) dt = \mbox{\rm Im}\, w(0)
\in  (0, \varepsilon_0) .
\end{eqnarray*}
Hence
\begin{eqnarray}
\label{AB}
&& \mbox{\rm Spec}(B + iA) = \left\{r + i\left(\sqrt{2\rho\varepsilon_0} + a_0\right)\right\} , \nonumber \\
&& \mbox{\rm Spec}(B - iA) = \mbox{\rm Spec}\left((B + iA)^*\right) = 
\left\{r - i\left(\sqrt{2\rho\varepsilon_0} + a_0\right)\right\} .
\end{eqnarray}

Consider the pencil 
$$
(\lambda I - (B + iA))(\lambda I - (B - iA)) = \lambda^2 I - 2 \lambda B + B^2 
-i [B, A] + A^2,
$$
where the square brackets denote the commutator. The Hermitian matrix $-i [B, A] + A^2$
is nonnegative. Indeed,
$$
i [B, A] = i [B, T_N(a)] \le  \left(2\|B\| \|T_N(a)\|\right) I \le 2\rho\varepsilon_0 I \le A^2 ,
$$
where the last inequality follows from the non-negativity of the Toeplitz matrix
$T_N(a)$ with the symbol $a \ge 0$. 

Let $R$ be the nonnegative square root of $-i [B, A] + A^2$.
Then
\begin{eqnarray*}
\det (\lambda^2 I - 2 \lambda B + B^2 + R^2) = \det\big((\lambda I - (B + iA))(\lambda I - (B - iA))\big)\\
= \det\big((\lambda I - (B + iA))\big)
\det\big((\lambda I - (B - iA))\big) .
\end{eqnarray*}
Hence it follows from \eqref{AB} that all roots of \eqref{det} belong to the interval 
$\{\lambda \in \mathbb{C} : \ \mbox{\rm Re}\, \lambda = r , \ 
|\mbox{\rm Im}\, \lambda| < \varepsilon\}$ provided 
$\sqrt{2\rho\varepsilon_0} + \varepsilon_0 < \varepsilon$.

It remains to estimate the norm of $R$.
\begin{eqnarray*}
\|Rx\|^2 = (R^2x, x) = ((-i [B, A] + A^2)x, x) \le 2\|B\|\|T_N(a)\| + \|A\|^2\\
\le 2\rho \varepsilon_0 + 
\left(\sqrt{2\rho\varepsilon_0} + \varepsilon_0\right)^2 , \ \ \ x \in \mathbb{C}^N, \ \ \|x\| = 1. 
\end{eqnarray*}
Choosing $\varepsilon_0 > 0$ such that the right-hand side is less than $\varepsilon^2$ we
get $\|R\| < \varepsilon$.
\end{proof}

\begin{remark}\label{TB}
{\rm Let
$$
T := \begin{pmatrix}
   B   &  R  \\
  R    &  0
\end{pmatrix} : \mathbb{C}^{2N} \to \mathbb{C}^{2N} .
$$
Then the set of the roots of \eqref{det} is equal to $\mbox{\rm Spec}_2(T, \mathbb{C}^N)$.
Since 
$$
\left\|T - \begin{pmatrix}
   B   &  0  \\
  0    &  0
\end{pmatrix}\right\| = \|R\| < \varepsilon ,
$$
$\mbox{\rm Spec}(T) \subset [\rho_- - \varepsilon, \rho_+ + \varepsilon]$ (see, e.g., 
\cite[Theorem V.4.10]{Kat}).}
\end{remark}

\begin{lemma}\label{ind}
Let $\varrho_-, \varrho_+ \in \mathbb{R}$ and let $T \in \mathbb{C}^{n\times n}$ be 
a Hermitian matrix such that 
$\mbox{\rm Spec}(T) \subset [\varrho_-, \varrho_+]$. Then
for any $\rho_- < \varrho_-$, any $\rho_+ > \varrho_+$ and any 
$r \in (\rho_-, \rho_+)$, $\delta, \varepsilon > 0$, 
one can choose $N, B$ and $R$ 
in Lemma \ref{block} in such a way that $N > n$ and
$$
B = \begin{pmatrix}
   T   &  S  \\
    S^*  &  K
\end{pmatrix} 
$$
with $\|S\|_{\mathbb{C}^{N - n} \to \mathbb{C}^n} < \delta$.
\end{lemma}
\begin{proof}
Let $\mu_1, \dots, \mu_n$ be the eigenvalues of $T$ repeated according to their multiplicities
and let $N \ge 2n$, $B', R'$ satisfy the conditions in Lemma \ref{block} 
with $\delta_0/(2n)$ in place of $\delta$, where $\delta_0 = \min\{\delta, \, \varrho_- - \rho_-, \,
\rho_+ - \varrho_+\}$.
The distance between any two consecutive distinct eigenvalues of $B'$ is less than $\delta_0/n$ as 
otherwise the distance from the centre of the interval between the eigenvalues to 
$\mbox{\rm Spec}(B')$ would have been greater than or equal to $\delta_0/(2n)$. Since the 
multiplicity of each $\mu_k$ is at most $n$, there exist distinct eigenvalues of $B'$ which
we denote by $\lambda_{\pm k}$, $k = 1, \dots, n$ and which satisfy the conditions
$$
\lambda_{-k} \le \mu_k \le \lambda_k \ \mbox{ and } \ \lambda_k - \lambda_{-k} < 2\delta .
$$
Then there exist $t_k \in [0, 1]$ such that $\mu_k = (1 - t_k)\lambda_{-k} + t_k \lambda_k$.
Let $u_m \in \mathbb{C}^N$, $m = \pm 1, \dots, \pm n$ be a normalised eigenvector of $B'$
corresponding to $\lambda_m$ and set
$$
v_k := \sqrt{1 - t_k}\, u_{-k} + \sqrt{t_k}\, u_k , \ \ 
v_{-k} := -\sqrt{t_k}\, u_{-k} + \sqrt{1 - t_k}\, u_k . 
$$
Since $\{u_{\pm k}\}_{k = 1}^n$ is an orthonormal set, $\|v_k\| = 1 = \|v_{-k}\|$,
$$
(v_k, v_{-k}) = -\sqrt{1 - t_k}\,\sqrt{t_k} + \sqrt{t_k}\,\sqrt{1 - t_k} = 0 ,
$$
and $(v_m, v_j) = 0$ if $m, j = \pm 1, \dots, \pm n$, $m \not= \pm j$. Hence
$\{v_{\pm k}\}_{k = 1}^n$ is an orthonormal set. Further,
\begin{eqnarray}\label{almdiag1}
(B'v_k, v_k) &=& (\sqrt{1 - t_k}\, \lambda_{-k} u_{-k} + \sqrt{t_k}\, \lambda_k u_k,
\sqrt{1 - t_k}\, u_{-k} + \sqrt{t_k}\, u_k) \nonumber \\
&=& (1 - t_k)\lambda_{-k} + t_k \lambda_k = \mu_k ,\\
(B'v_k, v_{-k}) &=& (\sqrt{1 - t_k}\, \lambda_{-k} u_{-k} + \sqrt{t_k}\, \lambda_k u_k,
-\sqrt{t_k}\, u_{-k} + \sqrt{1 - t_k}\, u_k)  \nonumber \\
&=& (\lambda_k - \lambda_{-k})\, \sqrt{1 - t_k}\,\sqrt{t_k} \in [0, \delta) , \nonumber 
\end{eqnarray} 
since $0 \le \sqrt{1 - t_k}\,\sqrt{t_k} \le 1/2$. It is also clear that
\begin{equation}
\label{almdiag2}
(B'v_k, v_m) = 0 , \ \ m \not= \pm k . 
\end{equation}
Let $U \in \mathbb{C}^{N\times N}$ be a unitary matrix such that
$$
U(\underbrace{0, \dots, 0, 1}_k, 0, \dots, 0)^T = \begin{cases}
v_k,  & k = 1. \dots, n, \\
 v_{n - k},   & k = n + 1, \dots, 2n .
\end{cases}
$$
Then
$$
U^*B'U = \begin{pmatrix}
\mbox{\rm diag}\{\mu_1, \dots, \mu_n\}   &  S'  \\
    (S')^*  &  K
\end{pmatrix} , 
$$
where $S' = \left(s_{kj}\right)_{n\times(N - n)}$, $|s_{kj}| < \delta$ if $j = n + k$, \
$s_{kj} = 0$ if $j \not= n + k$, \ $k = 1, \dots, n$, \ $j = n + 1, \dots, N$ (see \eqref{almdiag1},
\eqref{almdiag2}). It is easy to see that $\|S'\|_{\mathbb{C}^{N - n} \to \mathbb{C}^n} < \delta$.

Let $U_0 \in \mathbb{C}^{n\times n}$ be a unitary matrix such that 
$U_0TU_0^* = \mbox{\rm diag}\{\mu_1, \dots, \mu_n\}$, i.e. 
$U_0^*\,\mbox{\rm diag}\{\mu_1, \dots, \mu_n\}U_0 = T$, and let 
$$
U_1 := \begin{pmatrix}
   U_0   &  0  \\
   0   &  I_{N - n}
\end{pmatrix} .
$$
Then $U_1$ is a unitary matrix and
$$
U_1^*U^*B'UU_1 = \begin{pmatrix}
U_0^*\,\mbox{\rm diag}\{\mu_1, \dots, \mu_n\}U_0  &  U_0^*S'  \\
    (S')^*U_0  &  K
\end{pmatrix} = \begin{pmatrix}
   T   &  S  \\
    S^*  &  K
\end{pmatrix} , 
$$
where $S := U_0^*S'$. It is clear that 
$\|S\|_{\mathbb{C}^{N - n} \to \mathbb{C}^n} = 
\|S'\|_{\mathbb{C}^{N - n} \to \mathbb{C}^n} < \delta$.

Let
$$
B := V^*B'V, \ \ R := V^*R'V,
$$
where $V := UU_1$ is a unitary matrix. Then $B^* = B$, $R^* = R$, \ 
$\mbox{\rm Spec}(B) = \mbox{\rm Spec}(B')$, \ $\|R\| = \|R'\| < \varepsilon$, and
all zeros of the polynomial
\begin{eqnarray*}
\det (\lambda^2 I - 2 \lambda B + B^2 + R^2) &=& 
\det (V^*(\lambda^2 I - 2 \lambda B' + (B')^2 + (R')^2)V) \\
&=& \det (\lambda^2 I - 2 \lambda B' + (B')^2 + (R')^2)
\end{eqnarray*}
belong to the interval 
$\{\lambda \in \mathbb{C} : \ \mbox{\rm Re}\, \lambda = r , \ 
|\mbox{\rm Im}\, \lambda| < \varepsilon\}$.
\end{proof}

Let $T \in \mathbb{C}^{n\times n}$ and $m \le n$, $m \in \mathbb{N}$. Then for any
$\varepsilon > 0$ there exists $\Delta(T, m, \varepsilon) > 0$ such that for any
$D \in \mathbb{C}^{m\times m}$ with $\|D\| < \Delta(T, m, \varepsilon)$ the distance
from any root of the equation
$$
\det\left(P_m (\lambda I - T)^2|_{\mathbb{C}^m} + D\right) = 0
$$
to $\mbox{\rm Spec}_2(T, \mathbb{C}^m)$  is less than $\varepsilon$ and so is the distance from
any point of $\mbox{\rm Spec}_2(T, \mathbb{C}^m)$ to a root of this equation
(see, e.g.,  \cite[Theorem 4.10c]{Hen}). \\

We will use the following notation
$$
\ell^2(N) := \left\{x = (x_k)_{k \in \mathbb{N}} \in \ell^2 | \  x_k = 0, \ k > N\right\} \cong
\mathbb{C}^N 
$$
and will identify vectors $(x_1, \dots, x_N) \in \mathbb{C}^N$ with
$(x_1, \dots, x_N, 0, 0, \dots) \in \ell^2(N) \subset \ell^2$.
\begin{lemma}\label{mainl}
For any $\rho_- < \rho_+ \in \mathbb{R}$, $r \in (\rho_-, \rho_+)$ and 
any sequence $\alpha_l \in (0, 1)$, $l \in \mathbb{N}$ converging to $0$ 
there exist a self-adjoint operator $T \in \mathcal{B}(\ell^2)$ and
$N_l \in \mathbb{N}$, $l \in \mathbb{N}$ such that  
$\mbox{\rm Spec}(T) = [\rho_-, \rho_+]$, \ $N_l \uparrow +\infty$ as $l \uparrow +\infty$, and
\begin{equation}
\label{Spec2}
\mbox{\rm Spec}_2(T, \ell^2(N_l)) \subset \{\lambda \in \mathbb{C} : \ 
|\lambda - r| < \alpha_l\} , \ \ \forall l \in \mathbb{N} .
\end{equation}
\end{lemma}
\begin{proof}
Similarly to the proof of Lemma \ref{block} we can assume that $[\rho_-, \rho_+] = [-2, 2]$ 
as the general case can be
reduced to this one by dealing with 
$$\frac{4}{\rho_+ - \rho_-}\left(T - \frac{\rho_- + \rho_+}2\, I\right)
$$
 instead of $T$.

Let $\rho_0 = \varrho_0 = 0$, $\delta_0 = \varepsilon_0 = 1/4$, $N_0 = 1$, $B_0 = 0$, 
$T_0 = \begin{pmatrix}
0  &  0  \\
0  & 0 
\end{pmatrix}_{2\times 2}$. Using Lemmas \ref{block}, \ref{ind} and Remark \ref{TB} we
can successively construct $N_l$, $\delta_l = \varepsilon_l$, $B_l$, $T_l$ such that 
$B_l^* = B_l : \ell^2(N_l) \to \ell^2(N_l)$, $T_l^* = T_l : \ell^2(2N_l) \to \ell^2(2N_l)$,
$$
T_l := \begin{pmatrix}
   B_l   &  R_l  \\
  R_l    &  0
\end{pmatrix} ,
$$
$\|R_l\| < \varepsilon_l$,
$\mbox{\rm Spec}(B_l) \subset [-\rho_l, \rho_l]$, the distance from any point of $[-\rho_l, \rho_l]$
to $\mbox{\rm Spec}(B_l)$ is less than $\delta_l$, \ $\rho_l = 2 - 2^{-l}$,
\begin{equation}
\label{epsdel}
\delta_l = \varepsilon_l < \frac12\, \min\left\{\sqrt{\Delta(T_{l - 1}, N_{l - 1}, \alpha_{l - 1}/2)}\, ,\, 
\alpha_l , \, \varepsilon_{l - 1}\right\} ,
\end{equation}
\begin{eqnarray*}
& \mbox{\rm Spec}_2(T_l, \ell^2(N_l)) \subset \{\lambda \in 
\mathbb{C} : \ |\lambda - r| < \varepsilon_l\}
\subset \{\lambda \in \mathbb{C} : \ |\lambda - r| < \alpha_l/2\} , & \\
& B_{l + 1} = \begin{pmatrix}
   T_l   &  S_l  \\
    S_l^*  &  K_l
\end{pmatrix} , &
\end{eqnarray*}
$\|S_l\| < \delta_{l + 1}$, and $\mbox{\rm Spec}(T_l) \subset 
[-(\rho_l + \varepsilon_l), \rho_l + \varepsilon_l]
\subset [-\varrho_l, \varrho_l]$, $\varrho_l = 2 - 3\cdot2^{-l - 2} < \rho_{l + 1}$. The last inclusion
follows from \eqref{epsdel} as
\begin{equation}
\label{2}
\varepsilon_l < \frac{\varepsilon_{l - 1}}{2} < \cdots < \frac{\varepsilon_{0}}{2^l} = 2^{-l - 2} ,
\end{equation}
and 
$$
\rho_l + \varepsilon_l < 2 - 2^{-l} + 2^{-l - 2} = 2 - 3\cdot2^{-l - 2} = \varrho_l <
2 - 2^{-l - 1} =  \rho_{l + 1} .
$$

Since $T_l^* = T_l$ and $\mbox{\rm Spec}(T_l) \subset [-\varrho_l, \varrho_l]$,
\begin{equation}
\label{bound}
\|T_l\| \le \varrho_l = 2 - 3\cdot2^{-l - 2} < 2, \ \ \ \forall l \in \mathbb{N} .
\end{equation}

Let
$$
\widehat{B}_l := \begin{pmatrix}
   B_l   &  0  \\
   0   &  0
\end{pmatrix} : \ell^2 \to \ell^2 , \ \ \ 
\widehat{T}_l := \begin{pmatrix}
   T_l   &  0  \\
   0   &  0
\end{pmatrix} : \ell^2 \to \ell^2 .
$$
Suppose $x \in \ell^2(2 N_j)$, $j \le l$, $\|x\| \le 1$. Then
\begin{eqnarray*}
\left\|\widehat{T}_{l + 1}x - \widehat{T}_l x\right\| &\le& 
\left\|\widehat{T}_{l + 1}x - \widehat{B}_{l + 1} x\right\| +
\left\|\widehat{B}_{l + 1}x - \widehat{T}_l x\right\| =
\|R_{l + 1}x\| + \|S_l^*x\| \\
&<&  \varepsilon_{l + 1} + \delta_{l + 1} <
2^{-l - 3} + 2^{-l - 3} = 2^{-l - 2}
\end{eqnarray*} 
(see \eqref{2}), and therefore
\begin{eqnarray*}
\left\|\widehat{T}_{l + m}x - \widehat{T}_l x\right\| &\le& 
\sum_{p = 0}^{m - 1} \left\|\widehat{T}_{l + p + 1}x - \widehat{T}_{l + p} x\right\|
< \sum_{p = 0}^{m - 1} 2^{-l - p - 2} \\
&=& 2^{-l - 1} - 2^{-l - m - 1} < 2^{-l - 1} , \ \ \ m \in \mathbb{N} . 
\end{eqnarray*} 
Hence $\left(\widehat{T}_l x\right)_{l \in \mathbb{N}}$ is a convergent sequence in $\ell^2$ for 
any $x \in \ell^2(2N_j)$, $\forall j \in \mathbb{N}$. Since $\left\|\widehat{T}_l\right\| = 
\|T_l\| < 2$, $\forall l \in \mathbb{N}$ (see \eqref{bound}), the sequence 
$\left(\widehat{T}_l\right)_{l \in \mathbb{N}}$
is strongly convergent. Let $T \in \mathcal{B}(\ell^2)$ be its limit. Then $T^* = T$, $\|T\| \le 2$ and
\begin{equation}
\label{TTl}
\left\|Tx - \widehat{T}_l x\right\| \le 2^{-l - 1} , \ \ \ x \in \ell^2(2N_l), \ \ \|x\| \le 1 .
\end{equation}
Further, $\mbox{\rm Spec}(T) = [-2, 2]$. Indeed, take any $\lambda \in  [-2, 2]$. The distance 
from $\lambda$ to $\mbox{\rm Spec}(B_l)$ is less than $2^{-l} + \delta_l = 2^{-l} + \varepsilon_l <
2^{-l} + 2^{-l - 2}$ (see \eqref{2}). Using \cite[Theorem V.4.10]{Kat}) as in Remark \ref{TB}, one can
show that the distance 
from $\lambda$ to $\mbox{\rm Spec}(T_l)$ is less than $2^{-l} + 2^{-l - 2} + \varepsilon_l < 
2^{-l} + 2^{-l - 1}$. Hence there exists an eigenvector $x_l \in \ell^2(2N_l)$ of $T_l$
such that $\|x_l\| = 1$ and
$\|T_lx_l - \lambda x_l\| < 2^{-l} + 2^{-l - 1}$. It follows from \eqref{TTl} that
$$
\|Tx_l - \lambda x_l\| < 2^{-l} + 2^{-l - 1} + 2^{-l - 1} = 2^{-l + 1} , \ \ \ l \in \mathbb{N} .
$$
Therefore, $\lambda \in \mbox{\rm Spec}(T)$.

By construction, 
\begin{eqnarray*}
T_l x = P_{2N_l} B_{l + 1} x = P_{2N_l} P_{N_{l + 1}} T_{l + 1} x = 
P_{2N_l} T_{l + 1} x = P_{2N_l} P_{2 N_{l + 1}} T_{l + 2} x \\ 
= P_{2N_l} T_{l + 2} x = \cdots = P_{2N_l} T_{l + m} x = P_{2N_l} \widehat{T}_{l + m} x =
\cdots  , \ \ \ \forall x \in \ell^2(2N_l).
\end{eqnarray*}
So,
\begin{equation}
\label{restr}
P_{2N_l} T|_{\ell^2(2N_l)} = T_l , \ \ \ l \in \mathbb{N} .
\end{equation}
Let us now estimate the difference 
$$
P_{2N_l} T^2|_{\ell^2(2N_l)} - T_l^2 .
$$
Since
$$
T_{l + 1}^2 = \begin{pmatrix}
B_{l + 1}^2 + R_{l + 1}^2      &  B_{l + 1} R_{l + 1}  \\
   R_{l + 1} B_{l + 1}   &  R_{l + 1}^2
\end{pmatrix} , \ \ \ 
B_{l + 1}^2 = \begin{pmatrix}
T_l^2 + S_l S_l^*      &  T_l S_l + S_l K_l  \\
   S_l^* T_l + K_l S_l^*   &  S_l^* S_l + K_l^2
\end{pmatrix} ,
$$
we get
\begin{eqnarray*}
\left\|P_{2N_l} T_{l + 1}^2 x - T_l^2 x\right\| \le
\left\|P_{2N_l} T_{l + 1}^2 x - P_{2N_l} B_{l + 1}^2 x\right\| +
\left\|P_{2N_l} B_{l + 1}^2 x - T_l^2 x\right\| \\
= \|P_{2N_l} R_{l + 1}^2 x\| + \|S_l S_l^*x\| 
<  \varepsilon_{l + 1}^2 + \delta_{l + 1}^2 = 2 \varepsilon_{l + 1}^2 , \ \ \ 
x \in \ell^2(2N_l), \ \ \|x\| \le 1,
\end{eqnarray*} 
and therefore (see \eqref{2})
\begin{eqnarray*}
\left\|P_{2N_l} T_{l + m}^2 x - T_l^2 x\right\| \le
\sum_{p = 0}^{m - 1} \left\|P_{2N_l} T_{l + p + 1}^2 x - P_{2N_l}T_{l + p}^2 x\right\| \\
\le \sum_{p = 0}^{m - 1} \left\|P_{2N_{l + p}} T_{l + p + 1}^2 x - T_{l + p}^2 x\right\|
< 2 \sum_{p = 0}^{m - 1} \varepsilon_{l + p + 1}^2 \\
< 2 \varepsilon_{l + 1}^2 \sum_{p = 0}^{m - 1} \frac{1}{2^{2p}} 
= 2 \varepsilon_{l + 1}^2\, \frac{4}{3} \left(1 -  \frac{1}{2^{2m}}\right) < 
4 \varepsilon_{l + 1}^2 , \ \ \ m \in \mathbb{N} . 
\end{eqnarray*} 
Hence
\begin{equation}
\label{T2Tl2}
\left\|P_{2N_l} T^2|_{\ell^2(2N_l)} - T_l^2\right\| \le  4 \varepsilon_{l + 1}^2 <
\Delta(T_l, N_l, \alpha_l/2)
\end{equation}
(see \eqref{epsdel}). Finally,
\begin{eqnarray*}
&& P_{N_l} T|_{\ell^2(N_l)} = P_{N_l} T_l|_{\ell^2(N_l)} = B_l \ \mbox{ and } \\
&& \left\|P_{N_l} T^2|_{\ell^2(N_l)} - P_{N_l} T_l^2|_{\ell^2(N_l)}\right\| <
\Delta(T_l, N_l, \alpha_l/2) .
\end{eqnarray*}
Since $\mbox{\rm Spec}_2(T_l, \ell^2(N_l)) \subset 
\{\lambda \in \mathbb{C} : \ |\lambda - r| < \alpha_l/2\}$, 
\eqref{Spec2} follows from the definition of $\Delta(T_l, N_l, \alpha_l/2)$.
\end{proof}

\begin{remark}\label{Ni2Ni}
{\rm The proof of Lemma \ref{mainl} does not change if one adds the requirement
$$
\varepsilon_l < \frac12\, \sqrt{\Delta(T_{l - 1}, 2N_{l - 1}, \alpha_{l - 1})}
$$
to \eqref{epsdel}. Then
$$
\left\|P_{2N_l} T^2|_{\ell^2(2N_l)} - T_l^2\right\| \le  4 \varepsilon_{l + 1}^2 <
\Delta(T_l, 2N_l, \alpha_l)
$$
(see \eqref{T2Tl2}). Since $\mbox{\rm Spec}_2(T_l, \ell^2(2N_l)) =
\mbox{\rm Spec}(T_l)$, it follows from the definition of $\Delta(T_l, 2N_l, \alpha_l)$
and from what we know about $\mbox{\rm Spec}(T_l)$, that
$\mbox{\rm Spec}_2(T, \ell^2(2N_l))$ lies in an $\alpha_l$-neighbourhood of
$[-\varrho_l, \varrho_l]$ and the distance from any point of $[-2, 2]$ to 
$\mbox{\rm Spec}_2(T, \ell^2(2N_l))$ is less than $2^{-l} + 2^{-l - 1} + \alpha_l$.
Hence $\mbox{\rm Spec}_2(T, \ell^2(2N_l))$ converge to $[-2, 2]$ while 
$\mbox{\rm Spec}_2(T, \ell^2(N_l))$ converge to $\{r\}$ as $l \to +\infty$.}
\end{remark}

\section{Proof of Theorem \ref{Main}}
\begin{proof}
Let 
$$
r_j \in F\cap \left(\rho_-^{(j)}, \, \rho_+^{(j)}\right)  , \ \ \ j = 1, \dots, n
$$
and let $T^{(j)}$ and
$N_l^{(j)}$, $l \in \mathbb{N}$ be the same as in Lemma \ref{mainl} but with
$r_j \in \left(\rho_-^{(j)}, \, \rho_+^{(j)}\right)$ in place of $r \in (\rho_-, \rho_+)$.
Let $\mathcal{H}_j = \ell^2$, $\mathcal{H} = \bigoplus_{j = 1}^n \mathcal{H}_j$
and $T = \mbox{\rm diag}\{T^{(1)}, \dots, T^{(n)}\} \in \mathcal{B}(\mathcal{H})$.
It is clear that $T = T^*$ and $\mbox{\rm Spec}(T) = 
\bigcup_{j = 1}^n \left[\rho_-^{(j)}, \, \rho_+^{(j)}\right]$. 

Let $F_l$ be a finite subset of the interior of
$\mathcal{Q}\left(\bigcup_{j = 1}^n \left[\rho_-^{(j)}, \, \rho_+^{(j)}\right]\right)$ 
symmetric with respect to the real line and such that
\begin{equation}
\label{FlF}
d_H\left(F_l, F\right) < 2^{-l - 1} ,
\end{equation}
and let $F_l \cap \{\lambda \in \mathbb{C} : \ \mbox{\rm Im} \lambda \ge 0\} = 
\left\{\mu_1^{(l)}, \dots, \mu_{n_l}^{(l)}\right\}$.
For any $k = 1, \dots, n_l$ there exist $\lambda_{1, k}^{(l)},\, \lambda_{2, k}^{(l)},\, 
\lambda_{3, k}^{(l)} \in \bigcup_{j = 1}^n \left(\rho_-^{(j)}, \, \rho_+^{(j)}\right)$ such that
the convex hull of $\left\{\left(\mu_k^{(l)} - \lambda_{m, k}^{(l)}\right)^2\right\}_{m = 1}^3$
contains $0$, i.e.
\begin{eqnarray}\label{zeros}
& \exists\, t_{1, k}^{(l)},\, t_{2, k}^{(l)},\, t_{3, k}^{(l)} \in [0, 1] : \ \
t_{1, k}^{(l)} + t_{2, k}^{(l)} + t_{3, k}^{(l)} = 1, & \nonumber \\
& \sum_{m = 1}^3 t_{m, k}^{(l)} \left(\mu_k^{(l)} - \lambda_{m, k}^{(l)}\right)^2 = 0 &
\end{eqnarray}
(see \cite{S1}).

Let $\mathcal{L}_0 = \{0\}$, $\widetilde{N}_0 = 1$ and suppose we have constructed
$\mathcal{L}_0 \subset \mathcal{L}_1 \subset \cdots \subset \mathcal{L}_{l - 1} 
\subset \mathcal{H}$ and
$\widetilde{N}_0 < \widetilde{N}_1 < \cdots < \widetilde{N}_{l - 1} \in \mathbb{N}$ such that 
$\mathcal{L}_p \subseteq \bigoplus_{j = 1}^n \ell^2\left( \widetilde{N}_p\right)$, $p = 1, \dots, l - 1$.
Let us construct $\mathcal{L}_l$ and $\widetilde{N}_l$. Let $\widehat{N}_l^{(j)}$ be the
smallest number among $N_l^{(j)} < N_{l + 1}^{(j)} < N_{l + 2}^{(j)} < \cdots$ which is greater than
or equal to $\widetilde{N}_{l - 1}$. Then $\mathcal{L}_{l - 1} \subseteq \mathcal{L}^0_l :=
\bigoplus_{j = 1}^n \ell^2\left( \widehat{N}_l^{(j)}\right)$ and
\begin{equation}\label{rs}
d_H\left(\mbox{\rm Spec}_2(T, \mathcal{L}^0_l), \{r_1, \dots, r_n\}\right) < \alpha_l .
\end{equation}

Let $E(\cdot)$ be the spectral measure of $T$ and let 
$$
W_{m, k}^{(l)} \subset \bigcup_{j = 1}^n \left(\rho_-^{(j)}, \, \rho_+^{(j)}\right)
$$ 
be the $\varepsilon'_l$-neighbourhood of $\lambda_{m, k}^{(l)}$, where
$\varepsilon'_l$ is a small positive number to be specified later. Since the subspaces
$E\left(W_{m, k}^{(l)}\right)\mathcal{H} \subset \mathcal{H}$ are infinite dimensional, we can choose
vectors $u_{m, k}^{(l)} \in E\left(W_{m, k}^{(l)}\right)\mathcal{H}$ such that 
$\left\|u_{m, k}^{(l)}\right\| = 1$ and
\begin{eqnarray*}
& u_{m, k}^{(l)} \perp T^q\left(\mathcal{L}^0_l\right) , \ \ \   u_{m, k}^{(l)} \perp T^q u_{m', k'}^{(l)} , 
\ \ \ q = 0, 1, 2, &\\ 
& m, m' = 1, 2, 3, \ \ k, k' = 1, \dots n_l, \ \ (m, k) \not= (m', k') . & 
\end{eqnarray*}
Let
$$
v_k^{(l)} = \sum_{m = 1}^3 \sqrt{t_{m, k}^{(l)}}\, u_{m, k}^{(l)} ,  \ \ \ \ \
\mathcal{L}'_l = \mathcal{L}^0_l\oplus \mbox{\rm span} \left\{v_k^{(l)}\right\}_{k = 1}^{n_l}
$$
and let $\mathcal{P}_l^0 : \mathcal{H} \to \mathcal{L}^0_l$ and $\mathcal{P}'_l : 
\mathcal{H} \to \mathcal{L}'_l$ be the corresponding orthogonal projections.
Then $\left\|v_k^{(l)}\right\| = 1$, 
\begin{eqnarray*}
& v_k^{(l)} \perp T^q\left(\mathcal{L}^0_l\right) , \ \ \   v_k^{(l)} \perp T^q v_{k'}^{(l)} , 
\ \ \ q = 0, 1, 2, &\\ 
&  k, k' = 1, \dots n_l, \ \ k \not= k' , & 
\end{eqnarray*}
and 
\begin{eqnarray*}
&& \mathcal{P}'_l (\lambda I - T)^2|_{\mathcal{L}^0_l} = 
\mathcal{P}^0_l (\lambda I - T)^2|_{\mathcal{L}^0_l} ,  \\
&& \mathcal{P}'_l (\lambda I - T)^2 v_k^{(l)} =
\left((\lambda I - T)^2 v_k^{(l)}, v_k^{(l)}\right) v_k^{(l)} =: p_k^{(l)}(\lambda) v_k^{(l)} .
\end{eqnarray*}
Hence $\mathcal{P}'_l (\lambda I - T)^2|_{\mathcal{L}'_l}$ is unitarily equivalent to 
$$
\begin{pmatrix}
  \mathcal{P}^0_l (\lambda I - T)^2|_{\mathcal{L}^0_l}    &  0  \\
   0   &  \mbox{\rm diag}\left\{p_1^{(l)}(\lambda), \dots, p_{n_l}^{(l)}(\lambda)\right\}
\end{pmatrix}
$$
and
\begin{equation}
\label{unispec2}
\mbox{\rm Spec}_2(T, \mathcal{L}'_l) = \mbox{\rm Spec}_2(T, \mathcal{L}^0_l)
\bigcup \bigcup_{k = 1}^{n_l} \left\{\lambda \in \mathbb{C} : \ p_k^{(l)}(\lambda) = 0\right\} .
\end{equation}
By construction, the coefficients of the quadratic polynomial 
$$
p_k^{(l)}(\lambda) = \left((\lambda I - T)^2 v_k^{(l)}, v_k^{(l)}\right) = 
\sum_{m = 1}^3 t_{m, k}^{(l)} \left((\lambda I - T)^2 u_k^{(l)}, u_k^{(l)}\right)
$$ 
are real and
differ by less than $C \varepsilon'_l$ from those of
$$
q_k^{(l)}(\lambda) := \sum_{m = 1}^3 t_{m, k}^{(l)} \left(\lambda - \lambda_{m, k}^{(l)}\right)^2 .
$$
(It follows from the spectral theorem that one can take $C = 2\max\{1, \|T\|\}$.) 
Taking $\varepsilon'_l$ sufficiently small we can ensure that the zeros 
of $p_k^{(l)}(\lambda)$ differ from those of $q_k^{(l)}(\lambda)$ by less than 
$2^{-l -1}$. According to \eqref{zeros}, $\mu_k^{(l)}$ and its complex conjugate are the zeros
of $q_k^{(l)}(\lambda)$. Hence it follows from \eqref{FlF}, \eqref{rs} and \eqref{unispec2} that
$$
d_H\left(\mbox{\rm Spec}_2(T, \mathcal{L}'_l), F\right) < 
\max\left\{\alpha_l, 2^{-l}\right\} .
$$

Let $\widetilde{N}_l > \widetilde{N}_{l - 1}$,  $\widetilde{N}_l > \widehat{N}_l^{(j)}$,
$j = 1, \dots, n$, \ $P_{(l)} : \mathcal{H} \to \bigoplus_{j = 1}^n \ell^2\left( \widetilde{N}_l\right)$
be the orthogonal projection,
$$
\mathcal{L}_l = \mathcal{L}^0_l\oplus \mbox{\rm span} \left\{P_{(l)} v_k^{(l)}\right\}_{k = 1}^{n_l}
$$
and let $\mathcal{P}_l : \mathcal{H} \to \mathcal{L}_l$ be the corresponding orthogonal projection.
$\mbox{\rm Spec}_2(T, \mathcal{L}'_l)$ is the set of zeros of the determinant of a matrix
representation of $\mathcal{P}'_l (\lambda I - T)^2|_{\mathcal{L}'_l}$ which is a polynomial
in $\lambda$. If $\widetilde{N}_l$ is large, then $P_{(l)} v_k^{(l)}$ is  
close to $v_k^{(l)}$, and the coefficients of the polynomial corresponding to 
$\mathcal{P}_l (\lambda I - T)^2|_{\mathcal{L}_l}$ are close to their counterparts corresponding to
$\mathcal{P}'_l (\lambda I - T)^2|_{\mathcal{L}'_l}$. Hence taking $\widetilde{N}_l$ sufficiently 
large we get
$$
d_H\left(\mbox{\rm Spec}_2(T, \mathcal{L}_l), F\right) < 
\max\left\{\alpha_l, 2^{-l}\right\} 
$$
(see  \cite[Theorem 4.10c]{Hen}).
\end{proof}

\begin{remark}\label{osc}
{\rm $\mbox{\rm Spec}_2(T, \mathcal{L}_l)$ constructed in the above proof converge to $F$.
The limit behaviour of a sequence of second order relative spectra of $T$ may be considerably
more complicated than that. Let, for example, 
$F_0 \subseteq \mathcal{Q}\left(\bigcup_{j = 1}^n \left[\rho_-^{(j)}, \, \rho_+^{(j)}\right]\right)$ 
be another compact set symmetric with respect to the real line and such that
$$
F_0 \cap F\cap \left(\rho_-^{(j)}, \, \rho_+^{(j)}\right) \not= \emptyset , \ \ \ j = 1, \dots, n.
$$
Acting as in the proof above one can construct a sequence $(\mathcal{L}_{0,l})$ similar
to $(\mathcal{L}_l)$ and such that
$$
d_H\left(\mbox{\rm Spec}_2(T, \mathcal{L}_{0,l}), F_0\right) \to 0 \ \mbox{ as } \ l \to +\infty .
$$
Then it is easy to extract subsequences from $(\mathcal{L}_l)$ and
to $(\mathcal{L}_{0,l})$ and to combine them into a new sequence 
$(\mathcal{M}_l) \in \mathfrak{P}(\mathcal{H})$ in such a way that
$$
d_H\left(\mbox{\rm Spec}_2(T, \mathcal{M}_{2l}), F\right) \to 0 \ \mbox{ and } \
d_H\left(\mbox{\rm Spec}_2(T, \mathcal{M}_{2l + 1}), F_0\right) \to 0 \ \mbox{ as } \ l \to +\infty .
$$
One can of course carry out a similar procedure with more than just two limit sets $F$ and $F_0$.}
\end{remark}

\section{Concluding remarks}

The sequence $(N_l)$ in the proof of Lemma \ref{mainl} 
and $(\mbox{\rm dim}\, \mathcal{L}_l)$ in the proof of Theorem \ref{Main}  
are very rapidly increasing and it is not clear whether the above results have
serious implications for ``real life" computations involving second order relative
spectra. In all numerical examples studied so far (see, e.g., \cite{B1, B2, BB, BL, BS0, BS, LS, Str}), 
second order relative spectra seemed to approximate the whole spectrum quite well. 

{\bf Question:} {\it Can the phenomenon described by Lemma \ref{mainl} and Theorem \ref{Main} 
still happen if one restricts the rate of growth of $\mbox{\rm dim}\, \mathcal{L}_l$?} 

Note that
$$
{\lim_{N \to +\infty}} \!\!\! \empty^* \ \mbox{\rm Spec}_2(T, \ell^2(N)) \cap \mathbb{R} 
= [-2, 2] = \mbox{Spec}(T)
$$
in Remark \ref{Ni2Ni}. Here 
\begin{eqnarray*}
{\lim_{l \to +\infty}} \!\!\! \empty^* \
G_l := \Big\{z \in \mathbb{C}\, | \  \exists l_m \in \mathbb{N}, \
\exists z_{l_m} \in G_{l_m} : \
l_m \to +\infty \ \mbox{ and } \ z_{l_m} \to z  \\
\mbox{ as } \ m\to +\infty\Big\}  , \ \ \ G_l \subset \mathbb{C}, \ l \in \mathbb{N} .
\end{eqnarray*}

It is well known that
$$
\lim_{l \to +\infty}  \!\!\! \empty_* \ \mbox{Spec}(T, \mathcal{L}_l) \supseteq
\mbox{Spec}(T) ,  \ \ \ (\mathcal{L}_l) \in \mathfrak{P}(\mathcal{H}),
$$
where
$$
\lim_{l \to +\infty} \!\!\! \empty_* \
G_l := \left\{z \in \mathbb{C}\, | \ \exists z_l \in G_l : \
\lim_{l \to +\infty} z_l = z \right\}  
$$
(see, e.g., \cite{A} or \cite{S1}). It is reasonable therefore to use $\lim_*$ when
approximating $\mbox{Spec}(T)$ with the help of $\mbox{Spec}(T, \mathcal{L}_l)$.
On the other hand, the non-pollution result \eqref{nonpoll} shows it is more natural
to use $\lim^*$ when
approximating $\mbox{Spec}(T)$ with the help of $\mbox{Spec}_2(T, \mathcal{L}_l)$. \\

Another natural question is whether or not one can drop condition \eqref{nonempty}
in Theorem \ref{Main}.

{\bf Question:} {\it Can  the limit set of a sequence of second order relative spectra be disjoint from
the (essential) spectrum of $T = T^* \in \mathcal{B}(\mathcal{H})$?} \\

\textsc{Acknowledgement.}
I am grateful to Michael Strauss for very helpful comments and suggestions.

\end{document}